\documentclass[a4paper, 11pt, twoside]{article}

\usepackage{amsmath, amscd, amsfonts, amssymb, amsthm, latexsym}

\usepackage{graphicx}
\usepackage[left=1.3in, right=1.3in, top=1.5in, bottom=1.5in, includefoot, headheight=13.6pt]{geometry}

\input{xy}
\xyoption{all}

\newtheorem{theorem}{Theorem}[section]
\newtheorem{lemma}[theorem]{Lemma}
\newtheorem{corollary}[theorem]{Corollary}
\newtheorem{proposition}[theorem]{Proposition}

\theoremstyle{definition}
\newtheorem{definition}[theorem]{Definition}
\newtheorem{remark}[theorem]{Remark}

\newtheorem{conjecture}[theorem]{Conjecture}

\newcommand{\xysquare}[8]{
\[\xymatrix{
#1 \ar@{#5}[r] \ar@{#6}[d] & #2 \ar@{#7}[d]\\
#3 \ar@{#8}[r] & #4
}\]
}

\newcommand{\bb}{\mathbb}

\newcommand{\bor}{\partial}

\newcommand{\comment}[1]{}

\newcommand{\into}{\hookrightarrow}
\newcommand{\isoto}{\stackrel{\simeq}{\to}}
\newcommand{\Isoto}{\stackrel{\simeq}{\longrightarrow}}
\newcommand{\Mapsto}{\longmapsto}

\newcommand{\onto}{\twoheadrightarrow}
\newcommand{\Onto}{-\!\!\!\to\!\!\!\!\to}
\newcommand{\op}{\operatorname}
\newcommand{\pid}[1]{\langle #1 \rangle}

\newcommand{\roi}{\mathcal{O}}

\newcommand{\sub}[1]{{\mbox{\scriptsize #1}}}

\newcommand{\To}{\longrightarrow}
\newcommand{\ul}[1]{\underline{#1}}

\newcommand{\xto}{\xrightarrow}

\renewcommand{\cal}{\mathcal}

\renewcommand{\frak}{\mathfrak}
\newcommand{\indlim}{\varinjlim}
\renewcommand{\tilde}{\widetilde}

\renewcommand{\ker}{\operatorname{Ker}}
\renewcommand{\projlim}{\varprojlim}

\DeclareMathOperator{\codim}{codim}

\DeclareMathOperator{\Hom}{Hom}

\DeclareMathOperator{\Spec}{Spec}

\DeclareMathOperator{\Tor}{Tor}

\newcommand{\CH}{C\!H}
\newcommand{\BS}{B\!S}

\usepackage[ps2pdf=true,colorlinks,breaklinks=true]{hyperref}
\usepackage[figure,table]{hypcap}
\hypersetup{
	bookmarksnumbered,
	pdfstartview={FitH},
	citecolor={black},%{blue},
	linkcolor={black},%{red},
	urlcolor={black},
	pdfpagemode={UseOutlines}
}
\makeatletter
\newcommand\org@hypertarget{}
\let\org@hypertarget\hypertarget
\renewcommand\hypertarget[2]{%
  \Hy@raisedlink{\org@hypertarget{#1}{}}#2%
} 
\makeatother

\usepackage{fancyhdr}

\usepackage{sectsty}
\sectionfont{\Large\sc\centering}
\chapterfont{\large\sc\centering}
\chaptertitlefont{\LARGE\centering}
\partfont{\centering}

\begin{document}
\title{\vspace{-1cm}Zero cycles on singular varieties and their desingularisations}
\author{Matthew Morrow}
%\classno{}
\date{}
%\extraline{}

\maketitle

\begin{abstract}
We use pro cdh-descent of $K$-theory to study the relationship between the zero cycles on a singular variety $X$ and those on its desingularisation $X'$. We prove many cases of a conjecture of S.~Bloch and V.~Srinivas, and relate the Chow groups of $X$ to the Kerz--Saito Chow group with modulus of $X'$ relative to its exceptional fibre.

%Key words: $K$-theory, excision, descent, cyclic homology, zero cycles.
%MSC: 19D55 (primary), 13D03 14J17 14C35 (secondary).
\end{abstract}

\thispagestyle{empty}

\setcounter{section}{-1}
\section{Introduction}
Let $X'\to X$ be a desingularisation of a $d$-dimensional, integral variety over a field $k$, with exceptional fibre $E\into X$. Letting $rE$ denote the $r^\sub{th}$ infinitesimal thickening of $E$, we denote by $F^dK_0(X',rE)$ the subgroup of the relative $K$-group $K_0(X',rE)$ generated by the cycle classes of closed points of $X'\setminus E$, for each $r\ge1$. This inverse system \[F^dK_0(X',E)\longleftarrow F^dK_0(X',2E)\longleftarrow F^dK_0(X',3E)\longleftarrow\cdots\] was first studied by S.~Bloch and V.~Srinivas \cite{Srinivas1985a}, in the case of normal surfaces, as a means of relating zero cycles on the singular variety $X$ to zero cycles on the smooth variety $X'$. They conjectured [pg.~6, op.~cit.] in 1985 that this inverse system would eventually stabilise, i.e., $F^dK_0(X',rE)\isoto F^dK_0(X',(r-1)E)$ for $r\gg1$, with stable value equal to $F^dK_0(X)$, the subgroup of $K_0(X)$ generated by cycle classes of smooth, closed points of $X$.

The Bloch--Srinivas conjecture was proved for normal surfaces by A.~Krishna and Srinivas \cite[Thm.~1.1]{Krishna2002}, and later extended to higher dimensional, Cohen--Macaulay varieties with isolated singularities in characteristic zero by Krishna \cite[Thm.~1.1]{Krishna2006} \cite[Thm.~1.2]{Krishna2010}. The conjecture has not been previously verified in any case of non-isolated singularities, nor for any higher dimensional varieties in finite characteristic.

The primary goal of this paper is to prove the following cases of the Bloch--Srinivas conjecture for varieties which are regular in codimension one:

\begin{theorem}\label{theorem_intro}
Let $\pi:X'\to X$ be a desingularisation of a $d$-dimensional, quasi-projective, integral variety $X$ over an infinite, perfect field $k$ which is assumed to have strong resolution of singularities. Let $E\into X$ be a closed embedding covering the exceptional fibre, and assume that $\codim(X,\pi(E))\ge2$.

Then the associated Bloch--Srinivas conjecture is
\begin{enumerate}\itemsep0pt
\item true up to $(d-1)!$-torsion;
\item true if $X$ is projective, $k=k^\sub{alg}$, and $\op{char}k=0$;
\item true if $X$ is projective, $k=k^\sub{alg}$, and $d\le\op{char}k\neq0$;
\item true if $X$ is affine and $k=k^\sub{alg}$;\label{cone}
\item true ``up to a finite group'' if $k=k^\sub{alg}$ and $X_\sub{sing}$ is contained in an affine open of~$X$;
\item true if $\pi(E)$ is finite;
\item true if the cycle class map $\CH_0(X)\to F^dK_0(X)$ is an isomorphism.
\end{enumerate}
\end{theorem}

The group $\CH_0(X)$ appearing in part (vii) of Theorem \ref{theorem_intro} is the Levine--Weibel Chow group of zero cycles of the singular variety $X$ \cite{Levine1985b, Levine1985a}; it will be reviewed in Section \ref{definition_LW}.

Part (iv) of the Theorem, combined with arguments of Krishna \cite{Krishna2010} and R.~Murthy \cite{Murthy1994}, has concrete applications to Chow groups of cones and to the structure of modules and ideals of graded algebras; see Theorem \ref{theorem_Srinivas_2} and Corollaries \ref{corol_affine0} and \ref{corol_affine}.

This paper is intended partly to justify the author's pro cdh-descent theorem for $K$-theory \cite{Morrow_pro_cdh_descent}; indeed, the results of Theorem \ref{theorem_intro} are obtained in Section \ref{subsection_BS} as corollaries of the following general result, which itself is an immediate consequence of pro cdh-descent:

\begin{theorem}
Let $\pi:X'\to X$ be a desingularisation of a $d$-dimensional, quasi-projective, integral variety over an infinite, perfect field $k$ which is assumed to have strong resolution of singularities. Let $E\into X$ be a closed embedding covering the exceptional fibre. Then:
\begin{enumerate}
\item There exists a unique homomorphism $\BS_r:F^dK_0(X',rE)\to F^dK_0(X)$ for ${r\gg1}$ which is compatible with cycle classes of closed points.
\item The associated Bloch--Srinivas conjecture is true if and only if the canonical map $F^dK_0(X,rY)\to F^dK_0(X)$ is an isomorphism for $r\gg1$, where $Y:=\pi(E)_\sub{red}$.
\end{enumerate}
\end{theorem}

Section \ref{subsection_modulus} concerns Chow groups of zero cycles with modulus. If $X$ is a smooth, projective variety over a field $k$ and $D$ is an effective divisor on $X$, then the Chow group with modulus $\CH_0(X;D)$ is defined to be the free abelian group on the closed points of $X\setminus D$, modulo rational equivalence coming from closed curves $C$ which are not contained in $|D|$ and rational functions $f\in k(C)^\times$ which are $\equiv1$ mod $D$. This Chow group is central in M.~Kerz and S.~Saito's \cite{KerzSaito2013} higher dimensional class field theory.

It is natural to formulate an analogue of the Bloch--Srinivas conjecture for the Chow groups with modulus given by successive thickenings of the exceptional fibre of a desingularisation. We will explain this further in Section \ref{subsection_modulus}, where we prove it in the following cases:

\begin{theorem}\label{theorem_intro_KS}
Let $\pi:X'\to X$ be a desingularisation of a $d$-dimensional, quasi-projective, integral variety over an algebraically closed field $k$ which is assumed to have strong resolution of singularities. Let $D$ be an effective Cartier divisor on $X$ covering the exceptional fibre, and assume that $\codim(X,\pi(D))\ge2$.

Then the inverse system \[\CH_0(X';D)\longleftarrow \CH_0(X';2D)\longleftarrow \CH_0(X';3D)\longleftarrow\cdots\] eventually stabilises with stable value equal to $\CH_0(X)$, assuming that either
\begin{enumerate}\itemsep0pt
\item $X$ is projective and $\op{char}k=0$; or
\item $X$ is projective and $d\le\op{char}k\neq0$; or
\item $X$ is affine.
\end{enumerate}
\end{theorem}

Whenever the assertions of Theorem \ref{theorem_intro_KS} can be proved for a singular, projective variety $X$ over a {\em finite} field (e.g., for surfaces, as we shall see in Remark \ref{remark_surfaces}), it has applications to the class field theory of $X$; in particular, it shows that there is a reciprocity isomorphism of finite groups $\CH_0(X)^0\isoto\pi_1^\sub{ab}(X_\sub{reg})^0$. See Remark \ref{remark_normal_CFT} for further details. 

We prove Theorem \ref{theorem_intro_KS} by reducing it to the analogous assertion in $K$-theory, which is precisely the Bloch--Srinivas conjecture, and then applying Theorem \ref{theorem_intro}. This reduction is through the construction of a new cycle class homomorphism \[\CH_0(X;D)\To F^dK_0(X,D),\] which is valid for any effective Cartier divisor $D$ on a smooth variety $X$. This also allows us to prove the following result, which appears related to a special case of a conjecture of Kerz and Saito \cite[Qu.~V]{KerzSaito2013}:

\begin{theorem}
With notation and assumptions as in Theorem \ref{theorem_intro_KS}, the cycle class homomorphism \[\CH_0(X';rD)\To F^dK_0(X';rD)\] is an isomorphism for $r\gg1$.
\end{theorem}

\subsection*{Notation, conventions, etc.}
A field $k$ will be called {\em good} if and only if it is infinite, perfect, and has strong resolution of singularities, e.g., $\op{char}k=0$ suffices. A {\em $k$-variety} means simply a finite type $k$-scheme; further assumptions will be specified when required, and the reference to $k$ with occasionally be omitted. Our conventions about ``desingularisations'' can be found at the start of Section \ref{subsection_BS}.

A {\em curve} over $k$ is a one-dimensional, integral $k$-variety. Given a closed point $x\in C_0$, there is an associated order function $\op{ord}_x:k(X)^\times\to\bb Z$ characterised by the property that $\op{ord}_x(t)=\op{length}_{\roi_{C,x}}(\roi_{C,x}/t\roi_{C,x})$ for any non-zero $t\in\roi_{C,x}$; when $C$ is smooth $\op{ord}_x$ is the usual valuation associated to $x$.

An {\em effective divisor} $D$ on $X$ is by definition a closed subscheme whose defining sheaf of ideals $\roi_X(-D)$ is an invertible $\roi_X$-module, or, equivalently, is locally defined by a single non-zero-divisor; its associated support is denoted by $|D|$, but we write $X\setminus D$ in place of $X\setminus |D|$ for simplicity.

Given a closed embedding $Y=\Spec\roi_X/\cal I\into X$, its $r^\sub{th}$ infinitesimal thickening is denoted by $rY=\Spec\roi_X/\cal I^r$.

A {\em pro abelian group} $\{A_r\}_r$ is an inverse system of abelian groups, with morphisms given by the rule \[\Hom_{\op{Pro}Ab}(\{A_r\}_r,\{B_s\}_s):=\projlim_s\indlim_r\Hom_{Ab}(A_r,B_s).\] The category of pro abelian groups is abelian; we refer to \cite[App.]{ArtinMazur1969} for more details.

\subsection*{Acknowledgments}
Section \ref{section_desings} would not have been possible without discussions with V.~Srinivas and M.~Levine about zero cycles. Section \ref{subsection_modulus} was inspired by conversations with F.~Binda and S.~Saito at the \'Etale and motivic homotopy theory workshop in Heidelberg, 24--28 March 2014, and I thank A.~Schmidt and J.~Stix for organising such a pleasant event.

\section{Zero cycles of desingularisations}\label{section_desings}
In this section we prove cases of the Bloch--Srinivas conjecture relating zero cycles on a singular variety to those on its desingularisation.

There will be an important distinction between closed subsets $S\subseteq X$ and closed subschemes $Y\into X$; in an attempt to keep this clear we will use the differentiating notation $\subseteq$ and $\into$ just indicated. Any closed subscheme $Y\into X$ has an associated support $|Y|\subseteq X$, though we will continue to write $X\setminus Y$ rather than $X\setminus |Y|$ for the associated open complement, and any closed subset $S\subseteq X$ has an associated reduced closed subscheme $S_\sub{red}\into X$. The singular locus of $X$ is denoted by $X_\sub{sing}\subseteq X$.

\subsection{Review of the Levine--Weibel Chow group}\label{subsection_LW}
We begin by reviewing the Levine--Weibel Chow group of zero cycles \cite{Levine1985b, Levine1985a}, restricting to the situation that the singularities of $X$ are in codimension $\ge2$, since this is sufficient for our applications. Unless specified otherwise, $k$ is an arbitrary field.

\begin{definition}\label{definition_LW}
Let $X$ be an integral $k$-variety which is regular in codimension one, and $S\subseteq X$ any closed subset containing $X_\sub{sing}$. Then the associated {\em Levine--Weibel Chow group of zero cycles} is
\[\CH_0(X;S):=\frac{\mbox{free abelian group on closed points of $X\setminus S$}}{\pid{(f)_C\,:\,\mbox{$C\into X$ a curve not meeting $S$, and $f\in k(C)^\times$}}}\] where $(f)_C:=\sum_{x\in C_0}\op{ord}_x(f)\,x$ as usual. In particular, $\CH_0(X):=\CH_0(X;X_\sub{sing})$.
\end{definition}

\begin{remark}\label{remark_Levine_Weibel}
Several remarks should be made:
\begin{enumerate}
\item The group $\CH_0(X;S)$ we have just defined can actually only reasonably be called the Levine--Weibel Chow group of zero cycles if we assume that $\op{codim}(X,S)\ge2$. But it is convenient to introduce the notation in slightly greater generality since it will be useful in Section \ref{subsection_modulus}.
\item An inclusion of closed subsets $S\subseteq S'$ of $X$, both containing $X_\sub{sing}$, induces a canonical surjection $\CH_0(X;S')\onto\CH(X;S)$. This surjection is an isomorphism if $X$ is quasi-projective and $S$, $S'$ have codimension $\ge 2$, by a moving lemma \cite[pg.~113]{Levine1985a}.
\item Suppose that $X$ is a smooth $k$-variety and that $S\subseteq X$ is a closed subset. Then there is a canonical surjection $\CH_0(X;S)\onto\CH_0(X;\emptyset)=\CH_0(X)$, which will be an isomorphism if $S$ has codimension $\ge2$ and $X$ is quasi-projective, by the aforementioned moving lemma.
\item Suppose that $X'\to X$ is a proper morphism which restricts to an isomorphism $X'\setminus S'\isoto X\setminus S$ for some closed subsets $S\subseteq X$, $S'\subseteq X'$ containing the singular loci. Then the induced map $\CH_0(X;S)\to\CH_0(X';S')$ is an isomorphism. Indeed, both sides are generated by the closed points of $X'\setminus S'= X\setminus S$, and closed curves on $X$ not meeting $S$ correspond to closed curves on $X'$ not meeting $S'$.
\end{enumerate}
\end{remark}

To review the relationship between $\CH_0(X)$ and $K$-theory, we must first explain the cycle class map. Let $X$ be a $k$-variety, and $i:Y\into X$ a fixed closed subscheme. If $j:C\into X$ is a closed subscheme with image disjoint from both $|Y|$ and $X_\sub{sing}$, then $j$ is of finite Tor dimension since it factors as $C\into X_\sub{reg}\to X$, and it is moreover proper; thus the pushforward map $j_*:K(C)\to K(X)$ on the $K$-theory spectra is well-defined. Moreover, the projection formula \cite[Prop.~3.18]{Thomason1990} associated to the pullback diagram
\[\xymatrix{
\emptyset \ar[r]\ar[d] & C\ar[d]^j\\
Y\ar[r]_i & X
}\]
shows that the composition $K(C)\xto{j_*}K(X)\xto{i^*}K(Y)$ is null-homotopic, and thus there is an induced pushforward $j_*:K(C)\to K(X,Y)$. The {\em cycle class} of $C$ in $K_0(X,Y)$ is defined to be \[[C]:=j_*([\roi_C])\in K_0(X,Y).\] Although this appears to depend a priori on a chosen null-homotopy, it was shown by K.~Coombes \cite{Coombes1982} that the ``obvious choices of homotopies'' yield a class which is functorial with respect to both $X$ and $Y$, and so we will follow Coombes' choices. A codimension filtration on $K_0(X,Y)$ is now defined by
\begin{align*}
F^pK_0(X,Y):=\pid{[C]:C\into X&\mbox{ an integral closed subscheme of $X$ of codim}\ge p\\
&\mbox{ disjoint from }|Y|\mbox{ and }X_\sub{sing}}
\end{align*}
In particular, $F^dK_0(X,Y)$ is the subgroup of $K_0(X,Y)$ generated by the cycle classes of smooth, closed points of $X\setminus Y$. The following is standard:

\begin{lemma}\label{lemma_easy_Chern_class}
Let notation be as immediately above. If $j:C\into X$ is a closed embedding of a curve into $X$ not meeting $|Y|$ or $X_\sub{sing}$, and $f\in k(C)^\times$, then $\sum_{x\in C_0}\op{ord}_x(f)[x]=0$ in $K_0(X,Y)$.
\end{lemma}
\begin{proof}
One has $\sum_{x\in C_0}\op{ord}_x(f)[x]=j_*([\roi_C]-[f\roi_C])=j_*(0)=0$ .
\end{proof}

Now suppose that $X$ is a $d$-dimensional, integral $k$-variety which is regular in codimension one, let $Y\into X$ be a closed subscheme, and let $S\subseteq X$ be a closed subset containing both $|Y|$ and $X_\sub{sing}$. It follows from Lemma \ref{lemma_easy_Chern_class} that the cycle class homomorphism \[\CH_0(X;S)\To F^dK_0(X,Y),\quad x\Mapsto[x]\] is well-defined. In particular, taking $S=X_\sub{sing}$ and $Y=\emptyset$ yields the cycle class homomorphism \[[\;\;]:\CH_0(X)\To F^dK_0(X),\] which is evidently surjective. Moreover, as part of a general Riemann--Roch theory, M.~Levine \cite{Levine1983, Levine1985b} constructed a Chern class $ch_0:F^dK_0(X)\to \CH_0(X)$ such that the compositions $[\;\;]\circ ch_0$ and $ch_0\circ[\;\;]$ are both multiplcation by $(-1)^{d-1}(d-1)!$. In particular, $[\;\;]:\CH_0(X)\to F^dK_0(X)$ is an isomorphism if $d=2$.

We complete our review of the Levine--Weibel Chow group of zero cycles by presenting the higher dimensional cases in which the cycle class homomorphism can be shown to be an isomorphism:

\begin{theorem}[Barbieri Viale, Levine, Srinivas]\label{theorem_Levine}
Let $X$ be a $d$-dimensional, integral, quasi-projective variety over an algebraically closed field which is regular in codimension one. Then the cycle class homomorphism $\CH_0(X)\to F^dK_0(X)$ is
\begin{enumerate}\itemsep0pt
\item an isomorphism if $X$ is projective and $\op{char}k=0$;
\item an isomorphism if $X$ is projective and $d\le\op{char}k\neq0$;
\item an isomorphism if $X$ is affine and $\op{char}k$ is arbitrary;
\item a surjection with finite kernel if $X_\sub{sing}$ is contained in an affine open subscheme of $X$ and $\op{char}k=0$;
\item a surjection with finite kernel if $X_\sub{sing}$ is contained in an affine open subscheme of $X$ and $d\le \op{char}k\neq0$;
\end{enumerate}
\end{theorem}
\begin{proof}
Thanks to the existence of Levine's Chern class $ch_0$, it is enough to check that $\CH_0(X)$ has no $(d-1)!$-torsion in cases (i)--(ii), that it has only a finite amount of $(d-1)!$-torsion in cases (iv)--(v), and that it has no torsion in case (iii).

Then (i) and (ii) are \cite[Thm.~3.2]{Levine1985b}, while (iv) and (v) are \cite[Thm.~A]{BarbieriViale1992}. Finally, (iii) in characteristic zero (and when $d\le\op{char}k\neq0$) is \cite[Corol.~2.7]{Levine1985b}, and so it remains only to deal with the following case: assuming that $X$ is an integral, affine variety which is regular in codimension one, over an algebraically closed field of finite characteristic, we must show that $\CH_0(X)$ is torsion-free. This is true for the normalisation $\tilde X$ by \cite{Srinivas1989}, and so it remains only to check that $\CH_0(X)\isoto\CH_0(\tilde X)$. But since $X$ is assumed to be regular in codimension one, there are closed subsets $S\subseteq X$, $S'\subseteq \tilde X$ (given by the conductor ideal, for example) of codimension $\ge2$, containing the singular loci, and such that the morphism $\tilde X\to X$ restricts to an isomorphism $\tilde X\setminus S'\isoto X\setminus S$. Then, in the commutative diagram 
\[\xymatrix{
\CH_0(\tilde X;S') \ar[r] & \CH_0(\tilde X)\\
\CH_0(X;S) \ar[r]\ar[u] & \CH_0(X)\ar[u]
}\]
the horizontal arrows are isomorphisms by Remark \ref{remark_Levine_Weibel}(ii), while the left vertical arrow is an isomorphism by Remark \ref{remark_Levine_Weibel}(iv). Hence the right vertical arrow is an isomorphism, as required.
\end{proof}

\subsection{The Bloch--Srinivas conjecture}\label{subsection_BS}
Before we can carefully state the Bloch--Srinivas conjecture we must first fix some terminology concerning desingularisations. Given an integral variety $X$, a {\em desingularisation} is any proper, birational morphism $\pi:X'\to X$ where $X'$ is smooth; in particular, we allow the desingularisation to change the smooth locus of $X$, though it is not clear if this is ever important in practice. There exists a smallest closed subset $S\subseteq X$ with the property that $X'\setminus \pi^{-1}(S)\isoto X\setminus S$, and $\pi^{-1}(S)$ is known as the {\em exceptional set} of the resolution; setting $E:=\pi^{-1}(S)_\sub{red}$ yields the {\em exceptional fibre} $E\into X'$. Corollaries \ref{corollary_previous_results}--\ref{corollary_up_to_finite} will require that $\pi(|E|)$ has codimension $\ge2$ in $X$, which in particular implies that $X$ is regular in codimension one.

If $X'\to X$ is a desingularisation of an integral variety $X$, with exceptional fibre $E\into X'$, then Bloch and Srinivas \cite[pg.~6]{Srinivas1985a} made the following conjecture in 1985:

\begin{conjecture}[Bloch--Srinivas]\label{conjecture_BS}
The inverse system \[F^dK_0(X',E)\longleftarrow F^dK_0(X',2E)\longleftarrow F^dK_0(X',3E)\longleftarrow\cdots\] stabilises, with stable value $F^dK_0(X)$.
\end{conjecture}

\begin{remark}
To be precise, Bloch and Srinivas stated their conjecture in the case of a normal surface $X$ over an algebraically closed field, assuming that the desingularisation did not alter the smooth locus of $X$. If Conjecture \ref{conjecture_BS} is false because it has been formulated in excessive generality, it is the author's fault. In fact, we will consider Conjecture \ref{conjecture_BS} in greater generality still, by replacing the exceptional fibre $E$ by any reduced closed subscheme $E\into X'$ whose support contains the exceptional set (henceforth ``covers the exceptional set'').
\end{remark}

We interpret part of the Bloch--Srinivas conjecture as an implicit statement that there exists a cycle class homomorphism \[\BS_r:F^dK_0(X',rE)\To F^dK_0(X)\] for $r\gg1$ which is compatible with cycle classes of closed points $x\in X'\setminus E$, i.e., $\BS_r([x])=[x]$. Such a map $\BS_r$ is unique if it exists.

Our main technical theorem, which is an immediate consequence of the author's pro cdh-descent theorem for $K$-theory \cite{Morrow_pro_cdh_descent}, proves the existence of the maps $\BS_r$ in full generality, and reduces the Bloch--Srinivas conjecture to the study of the $K$-theory of $X$:

\begin{theorem}\label{theorem_BS}
Let $X$ be a $d$-dimensional, integral variety over a good field $k$; let $\pi:X'\to X$ be a desingularisation, $E\into X'$ any reduced closed subscheme covering the exceptional set, and set $Y:=\pi(|E|)_\sub{red}$. Then:
\begin{enumerate}
\item For $r\gg1$, the canonical map $F^dK_0(X,rY)\to F^dK_0(X)$ factors through the surjection $F^dK_0(X,rY)\to F^dK_0(X',rE)$, i.e., there exists a commutative diagram
\[\xymatrix@=1.5cm{
F^dK_0(X',rE) \ar@{-->}[dr]^{\exists\, \BS_r}\ar[r] & F^dK_0(X')\\
F^dK_0(X,rY)\ar@{->>}[u]\ar[r]& F^dK_0(X)\ar[u]
}\]
\item The following are equivalent:
\begin{enumerate}
\item The associated Bloch--Srinivas conjecture is true, i.e., $\BS_r$ is an isomorphism for $r\gg1$.
\item The canonical map $F^dK_0(X,rY)\to F^dK_0(X)$ is an isomorphism for $r\gg1$.
\item The canonical map $F^dK_0(X,rY)\to F^dK_0(X)$ is an isomorphism for all~${r\ge1}$.
\end{enumerate}
\end{enumerate}
\end{theorem}
\begin{proof}
There is an abstract blow-up square \[\xymatrix{
Y'\ar[r]\ar[d] & X' \ar[d]^\pi \\
Y\ar[r] & X
}\]
where $Y':=X'\times_XY$; note that $Y'$ is a nilpotent thickening of $E$. By pro cdh-descent for $K$-theory \cite[Thm.~0.1]{Morrow_pro_cdh_descent} (it is here that the field $k$ is required to be good), the canonical homomorphism of pro abelian groups \[\{K_0(X,rY)\}_r\To\{K_0(X',rY')\}_r\cong\{K_0(X',rE)\}_r\] is an isomorphism. Restricting to the codimension filtration we deduce that the homomorphism \[\{F^dK_0(X,rY)\}_r\To\{F^dK_0(X',rE)\}_r\tag{\dag}\] is injective; but each map $F^dK_0(X,rY)\to F^dK_0(X',rE)$ is evidently surjective, since both sides are generated by the closed points of $X\setminus Y=X'\setminus E$. Thus (\dag) is an isomorphism.

By definition of an isomorphism of pro abelian groups, this implies that for any $s\ge1$ there exists $r\ge s$ and a homomorphism $F^dK_0(X',rE')\to F^dK_0(X,sY)$ making the diagram commute:
\[\xymatrix@=1cm{
F^dK_0(X',rE) \ar@{-->}[dr]^{\exists}&\\
F^dK_0(X,rY)\ar@{->>}[u]\ar@{->>}[r]& F^dK_0(X,sY)
}\]
Note that the vertical and horizontal arrows are surjective, since the groups are generated by the closed points of $X\setminus Y=X'\setminus E$. This diagram shows that the canonical map $F^dK_0(X,rY)\to F^dK_0(X)$ factors through the surjection $F^dK_0(X,rY)\to F^dK_0(X',rE)$, proving (i).

This gives a commutative diagram
\[\xymatrix@=1cm{
F^dK_0(X',rE) \ar[dr]\ar@/^10pt/[drr]^{\BS_r}&\\
F^dK_0(X,rY)\ar@{->>}[u]\ar@{->>}[r]& F^dK_0(X,sY)\ar[r] & F^dK_0(X)
}\]
from which a simple diagram chase yields the following implications (valid for any $s\ge1$ and $r\gg s$):
\begin{quote}
$F^dK_0(X,rY)\to F^dK_0(X)$ is an isomorphism $\implies$ $\BS_r$ is an isomorphism.\\
$\BS_r$ is an isomorphism $\implies$ $F^dK_0(X,sY)\to F^dK_0(X)$ is an isomorphism.
\end{quote}
The equivalence of (a)--(c) follow, completing the proof.
\end{proof}

\begin{remark}
Suppose that the desingularisation $X'\to X$ does not change the smooth locus of $X$ and that $E$ is equal to the exceptional fibre (this is probably the most important case of the conjecture). Then Theorem \ref{theorem_BS} states that the associated Bloch--Srinivas conjecture is true if and only if $F^dK_0(X,rY)\isoto F^dK_0(X)$ for $r\gg1$, where $Y=(X_\sub{sing})_\sub{red}$.

In particular, under these additional hypotheses on $X'$ and $E$ we see that the Bloch--Srinivas conjecture depends only on $X$, and not on the chosen desingularisation. Even in the case of arbitrary desingularisations and general $E$ covering the exceptional set, Theorem \ref{theorem_BS} shows that the associated Bloch--Srinivas conjecture depends only on $X$ and $\pi(|E|)$.
\end{remark}

\begin{remark}\label{remark_stabilisation}
The proof of Theorem \ref{theorem_BS} shows the following: the inverse system $F^dK_0(X',rE)$, $r\ge1$, stabilises if and only if the inverse system $F^dK_0(X,rY)$, $r\ge~1$, stabilises, in which case the canonical map $F^dK_0(X,rY)\to F^dK_0(X',rE)$ is an isomorphism for $r\gg1$.
\end{remark}

The following corollary recovers all previously known cases of the Bloch--Srinivas conjecture (normal surfaces \cite[Thm.~1.1]{Krishna2002}; Cohen--Macaulay varieties with isolated singularities in characteristic zero \cite[Thm.~1.1]{Krishna2006} \cite[Thm.~1.2]{Krishna2010}; note that in these cases one can use the reduction ideal trick of Weibel \cite{Weibel2001} to avoid assuming that $k$ has resolution of singularities, c.f., Remark \ref{remark_surfaces}):

\begin{corollary}\label{corollary_previous_results}
Let $X$ be a $d$-dimensional, integral variety over a good field $k$; let $\pi:X'\to X$ be a desingularisation, and $E\into X'$ any reduced closed subscheme covering the exceptional set. Assume $\pi(|E|)$ is finite and $d\ge2$.

Then the associated Bloch--Srinivas conjecture is true.
\end{corollary}
\begin{proof}
Set $Y:=\pi(|E|)_\sub{red}$. According to Theorem \ref{theorem_BS}, it is necessary and sufficient to show that the canonical map $F^dK_0(X,rY)\to F^dK_0(X)$ is an isomorphism for all $r\ge1$. But this follows from \cite[Lem.~3.1]{Krishna2006} since $rY$ is zero dimensional.
\end{proof}

The next corollary proves the Bloch--Srinivas conjecture under the assumption that the cycle class homomorphism $\CH_0(X)\to F^dK_0(X)$ is an isomorphism:

\begin{corollary}\label{corollary_Chow_isom}
Let $X$ be a $d$-dimensional, integral, quasi-projective variety over a good field $k$; let $\pi:X'\to X$ be a desingularisation, and $E\into X'$ any reduced closed subscheme covering the exceptional set. Assume $\op{codim}(X,\pi(|E|))\ge2$ and that the cycle class map $\CH_0(X)\to F^dK_0(X)$ is an isomorphism.

Then the associated Bloch--Srinivas conjecture is true.
\end{corollary}
\begin{proof}
Set $Y=\pi(|E|)_\sub{red}$. According to Theorem \ref{theorem_BS}, it is necessary and sufficient to show that the canonical map $F^dK_0(X,rY)\to F^dK_0(X)$ is an isomorphism for all $r\ge1$. To prove this we consider the commutative diagram
\[\xymatrix{
F^dK_0(X,rY)\ar[r] & F^dK_0(X)\\
\CH_0(X;|Y|)\ar[r]\ar[u] & \CH_0(X)\ar[u]
}\]
The right vertical arrow is an isomorphism by assumption, the bottom horizontal arrow is an isomorphism by Remark \ref{remark_Levine_Weibel}(ii), and the left vertical arrow is a surjection since the domain and codomain are generated by the closed points of $X\setminus Y$. It follows that the top horizontal arrow (and left vertical arrow -- we will need this in the proof of Theorem \ref{theorem_modulus}) is an isomorphism, as desired.
\end{proof}

In particular, we have proved the Bloch--Srinivas conjecture for projective varieties over an algebraically closed field of characteristic zero which are regular in codimension one:

\begin{corollary}\label{corollary_algebraically_closed}
Let $X$ be a $d$-dimensional, integral variety over an algebraically closed field $k$ which has strong resolution of singularities; let $\pi:X'\to X$ be a desingularisation, and $E\into X'$ any reduced closed subscheme covering the exceptional set. Assume $\op{codim}(X,\pi(|E|))\ge2$ and that one of the following is true:
\begin{enumerate}\itemsep0pt
\item $X$ is projective and $\op{char}k=0$; or
\item $X$ is projective and $d\le\op{char}k\neq0$; or
\item $X$ is affine and $\op{char}k$ is arbitrary.
\end{enumerate}
Then the associated Bloch--Srinivas conjecture is true.
\end{corollary}
\begin{proof}
This follows from Corollary \ref{corollary_Chow_isom} and the results of Levine and Srinivas recalled in Theorem \ref{theorem_Levine}.
\end{proof}

\begin{remark}
It seems plausible that some descent or base change technique should eliminate the requirement in Corollary \ref{corollary_algebraically_closed} that $k$ be algebraically closed.
\end{remark}

We can also solve the Bloch--Srinivas conjecture up to $(d-1)!$-torsion whenever $X$ is regular in codimension one:

\begin{corollary}
Let $X$ be a $d$-dimensional, integral, quasi-projective variety over a good field $k$; let $\pi:X'\to X$ be a desingularisation, and $E\into X'$ any reduced closed subscheme covering the exceptional set. Assume $\op{codim}(X,\pi(|E|))\ge2$. 

Then the associated Bloch--Srinivas conjecture is true up to $(d-1)!$-torsion, i.e., the maps \[\BS_r:F^dK_0(X',rE)\otimes\bb Z[\tfrac{1}{(d-1)!}]\To F^dK_0(X)\otimes\bb Z[\tfrac{1}{(d-1)!}]\] are isomorphisms for $r\gg1$.
\end{corollary}
\begin{proof}
Set $Y=\pi(|E|)_\sub{red}$. By a trivial modification of Theorem \ref{theorem_BS}, it is necessary and sufficient to show that the canonical map $F^dK_0(X,rY)\to F^dK_0(X)$ is an isomorphism for all $r\ge1$ after inverting $(d-1)!$. This follows exactly as in Corollary \ref{corollary_Chow_isom}, since the cycle class map $\CH_0(X)\to F^dK_0(X)$ is an isomorphism after inverting $(d-1)!$, thanks to the existence of Levine Chern class $ch_0:F^dK_0(X)\to\CH_0(X)$.
\end{proof}

The next result solves the Bloch--Srinivas conjecture up to a finite group when the singular locus $X_\sub{sing}$ has codimension $\ge2$ and is contained in an affine open of $X$. Note that the ``obvious'' cases in which this happens, namely when $X_\sub{sing}$ is finite or $X$ itself is affine, are already largely covered by Corollaries \ref{corollary_previous_results} and \ref{corollary_algebraically_closed}(iii) respectively:

\begin{corollary}\label{corollary_up_to_finite}
Let $X$ be a $d$-dimensional, integral, quasi-projective variety over an algebraically closed field $k$ which has strong resolution of singularities; let $\pi:X'\to X$ be a desingularisation, and $E\into X'$ any reduced closed subscheme covering the exceptional set. Assume $\op{codim}(X,\pi(|E|))\ge2$, that $X_\sub{sing}$ is contained in an affine open of $X$, and moreover that $d\le\op{char} k$ if $\op{char} k\neq 0$.

Then the maps \[\BS_r:F^dK_0(X',rE)\To F^dK_0(X)\] are surjective with finite kernel for $r\gg1$, and the inverse system $F^dK_0(X',rE)$, $r\ge1$, stabilises.
\end{corollary}
\begin{proof}
We concatenate commutative diagrams we have already considered in Theorem \ref{theorem_BS} and Corollary \ref{corollary_Chow_isom}:
\[\xymatrix{
F^dK_0(X',rE) \ar@/^10pt/[dr]^{\BS_r}&\\
F^dK_0(X,rY)\ar@{->>}[u]\ar[r]& F^dK_0(X)\\
\CH_0(X;|Y|)\ar[r]^\cong\ar@{->>}[u] & \CH_0(X)\ar[u]
}\]
The left vertical arrows are surjective since the groups are generated by the closed points of $X\setminus Y=X'\setminus E$; the bottom horizontal arrow is an isomorphism by Remark \ref{remark_Levine_Weibel}(ii); the right vertical arrow is surjective with finite kernel $\Lambda$ by the result of Barbieri Viale recalled in Theorem \ref{theorem_Levine}.

A simple diagram chase shows that $\BS_r$ is surjective and that its kernel $\Lambda_r$ is naturally a quotient of $\Lambda$. Since $\Lambda$ is finite, this tower of quotients $\Lambda_r$ must eventually stabilise, completing the proof.
\end{proof}

\begin{remark}
We finish our discussion of the Bloch--Srinivas conjecture with a remark about $SK_1$. Let $\pi:X'\to X$, $E$, $Y$, $k$ be as in the statement of Theorem \ref{theorem_BS}, and assume $X$ is quasi-projective and $\op{codim}(X,Y)\ge2$.

The maps $F^dK_0(X,rY)\to F^dK_0(X)$ are surjective for all $r\ge1$ (by Remark \ref{remark_Levine_Weibel}(ii) and existence of the cycle class maps); hence we may add \begin{enumerate}\item[\em(b$^\prime$)] {\em The canonical map $F^dK_0(X,rY)\to F^dK_0(X)$ is injective for $r\gg1$.}\end{enumerate} to the list of equivalent conditions in Theorem \ref{theorem_BS}(ii).

Next, it follows from \cite[Lem.~3.1]{Krishna2006} that (b$^\prime$) (hence the associated Bloch--Srinivas conjecture) would follow from showing that $\bor(SK_1(rY))=0$, where $\bor:K_1(rY)\to K_0(X,rY)$ is the boundary map and $SK_1(rY):=\ker(K_1(rY)\onto H^0(rY,\roi_{rY}^\times))$; equivalently, it is enough to show that $SK_1(X)\to SK_1(rY)$ is surjective. Using the arguments of Theorem \ref{theorem_BS} it would even be enough to show, for each $r\gg1$, that \[\op{Im}(SK_1(sY)\to SK_1(rY))\subseteq\op{Im}(SK_1(X)\to SK_1(rY))\] for some $s\ge r$. It is not clear whether one should expect this to be true.
\end{remark}

We finish the section with some consequence of the Bloch--Srinivas conjecture. The following result about Chow groups of cones was conjectured by Srinivas \cite[\S3]{Srinivas1987a} in 1987; it was proved by Krishna \cite[Thm.~1.5]{Krishna2010} under the assumption that the cone $X$ was normal and Cohen--Macaulay, and we will combine his argument with Theorem \ref{theorem_BS} to establish the result in full generality; due to the failure of Kodaira vanishing in finite characteristic we must restrict to characteristic zero:

\begin{theorem}\label{theorem_Srinivas_2}
Let $Y\into\bb P_k^N$ be a $d$-dimensional, smooth, projective variety over an algebraically closed field $k$ of characteristic zero; assume $d>0$ and $H^d(Y,\roi_Y(1))=0$, and let $X$ be the affine cone over $Y$. Then $\CH_0(X)=0$.
\end{theorem}
\begin{proof}
We may resolve $X$, which has a unique singular point, to obtain $X'$ which is a line bundle over over $Y$, of which the zero section is the exceptional fibre of the resolution $X'\to X$. By Corollary \ref{corollary_previous_results} or \ref{corollary_algebraically_closed}(iii), we know that $\CH_0(X)\cong F^{d+1}K_0(X',rY)$ for $r\gg1$; moreover, $\CH_0(X')$ surjects onto $F^dK_0(X')$, and $\CH_0(X')=0$ since $X'$ is a line bundle, so $F^dK_0(X')=0$. So it is enough to show that the canonical map $F^{d+1}K_0(X',rY)\to F^{d+1}K_0(X')$ is an isomorphism. According to Krishna's proof of \cite[Cor.~8.5]{Krishna2010}, this would follows from knowing that:
\begin{enumerate}
\item $H^d(X',\cal K_{d,X'})\otimes k^\times\To H^d(Y,\cal K_{d,Y})\otimes k^\times$ is surjective; and
\item $H^d\left(rY,\frac{\Omega^d_{(rY,Y)}}{d\Omega^{d-1}_{(rY,Y)}}\right)=0$ for $r\gg1$.
\end{enumerate}
Condition (i) is satisfied since the zero section $Y\into X'$ is split by the line bundle structure map $X'\to Y$. Condition (ii) is deduced from the Akizuki--Nakano vanishing theorem, as explained in Lem.~9.1 and the proof of Thm.~1.5 in \cite{Krishna2010}.
\end{proof}

\begin{corollary}\label{corol_affine0}
Let $Y,k$ be as in the previous theorem, and let $A$ be its homogeneous coordinate ring. Then every projective module over $A$ of rank at least $d$ has a free direct summand of rank one.
\end{corollary}
\begin{proof}
This follows from Theorem \ref{theorem_Srinivas_2} using a result of R.~Murthy \cite[Cor.~3.9]{Murthy1994}.
\end{proof}

\begin{corollary}\label{corol_affine}
Let $k$ be an algebraically closed field of characteristic zero, and $f\in k[\ul t]:=k[t_0,\dots,t_d]$ a homogenous polynomial of degree at most $d+1$ which defines a smooth hypersurface in $\bb P_k^d$. Then every smooth closed point of $\Spec k[\ul t]/\pid f$ is a complete intersection.

In other words, if $\frak m$ is any maximal ideal of $k[\ul t]$ containing $f$ other than the origin, then $\frak m=\pid{f,f_1,\dots,f_d}$ for some $f_1,\dots,f_d\in k[\ul t]$.
\end{corollary}
\begin{proof}
This also follows from Theorem \ref{theorem_Srinivas_2} thanks to Murthy \cite[Thm.~4.4]{Murthy1994}.
\end{proof}

\section{Chow groups with modulus}\label{subsection_modulus}
If $X$ is a smooth variety over a field $k$, and $D$ is an effective divisor on $X$, then the Chow group $\CH_0(X;|D|)$ from Definition \ref{definition_LW} may be a rather coarse invariant, as there may not be enough curves on $X$ avoiding the codimension-one subset $|D|$. Of greater interest is $\CH_0(X;D)$, the Chow group of zero cycles on $X$ with modulus $D$, which we will define precisely in Definition \ref{definition_modulus}; note the notational difference, indicating that $\CH_0(X;D)$ depends not only on the support of $D$, but on its schematic, and possibly non-reduced, structure.

According to the higher dimensional class field theory of M.~Kerz and S.~Saito, when $k$ is finite and $X$ is proper over $k$, the group $\CH_0(X;D)$ classifies the abelian \'etale covers of $X\setminus D$ whose ramification is bounded by $D$; we refer the reader to \cite{KerzSaito2013} for details since we will not require any of their results.

We now turn to definitions, and refer again to [op.~cit.] for a more detailed exposition. Let $C$ be a smooth curve over a field $k$, and $D$ an effective divisor on $C$; writing $D=\sum_{x\in |D|}m_xx$ as a Weil divisor, we let \[k(C)_D^\times:=\{f\in k(C)^\times:\op{ord}_x(f-1)\ge m_x\mbox{ for all }x\in |D|\}\] denote the rational functions on $C$ which are $\equiv 1$ mod $D$. More generally, if $X$ is a smooth variety over $k$ and $D$ is an effective divisor on $X$, then for any curve $C\into X$ which is not contained in $|D|$ we write \[k(C)_D^\times:=k(\tilde C)^\times_{\phi^*D},\] where $\phi:\tilde C\to C\into X$ is the resulting map from the normalisation $\tilde C$ to $X$; evidently $k(C)_D^\times=k(C)^\times$ if $C$ does not meet $|D|$.

The Chow group with modulus is defined as follows:

\begin{definition}\label{definition_modulus}
Let $X$ be a smooth variety over $k$, and $D$ an effective divisor on $X$. Then the associated {\em Chow group of zero cycles of $X$ with modulus $D$} is \[\CH_0(X;D):=\frac{\mbox{free abelian group on closed points of $X\setminus D$}}{\pid{(f)_C\,:\,\mbox{$C\into X$ a curve not contained in $|D|$, and $f\in k(C)_D^\times$}}}\] where $(f)_C=\sum_{x\in C_0}\op{ord}_x(f)\,x$.
\end{definition}

If we were to define \[k(C)_{|D|}^\times:=\begin{cases}k(C)^\times &\mbox{if $C$ does not meet $|D|$,}\\1&\mbox{if $C$ meets $|D|$,}\end{cases}\] and repeat Definition \ref{definition_modulus} with $|D|$ in place of $D$, then the resulting group $\CH_0(X;|D|)$ would coincide with that defined in Definition \ref{definition_LW}. Since $k(C)^\times_{|D|}\subseteq k(C)^\times_D$, we thus obtain a canonical surjection \[\CH_0(X;|D|)\Onto\CH_0(X;D).\] One sense in which $\CH_0(X;D)$ is a more refined invariant than $\CH_0(X;|D|)$ is that the cycle class homomorphism $\CH_0(X;|D|)\to K_0(X,D)$ of Section \ref{subsection_LW} factors through $\CH_0(X;D)$. There does not appear to be a proof of this important result in the literature, so we give one here, beginning with a much stronger result in the case of curves:

\begin{lemma}\label{lemma_class_group_of_curve}
Let $C$ be a smooth curve over a field $k$, and $D$ an effective divisor on $C$. Then the canonical map \[\mbox{\em free abelian group on closed points of }C\setminus D\;\To\; K_0(C,D),\quad x\Mapsto[x]\] induces an injective cycle class homomorphism \[\CH_0(C;D)\To K_0(C,D),\] which is an isomorphism if $D\neq 0$ (and has cokernel $=\bb Z$ if $D=0$).
\end{lemma}
\begin{proof}
The Zariski descent spectral sequence for the $K$-theory of $C$ relative to $D$ degenerates to short exact sequences, since $\dim C=1$, yielding in particular \[0\To H^1(C,\cal K_{1,(C,D)})\To K_0(C,D)\To H^0(C,\cal K_{0,(C,D)})\To 0.\] Here $\cal K_{i,(C,D)}$ is by definition the Zariski sheafification on $C$ of the presheaf $U\mapsto K_i(U,U\times_CD)$.

To describe these terms further we make some standard comments about the long exact sequence of sheaves \[\cal K_{2,C}\to\cal K_{2,D}\to\cal K_{1,(C,D)}\to\cal K_{1,C}\to\cal K_{1,D}\to\cal K_{0,(C,D)}\to \cal K_{0,C}\to \cal K_{0,D}.\] Firstly, $\cal K_{1,C}\cong\roi_C^\times$ and $\cal K_{1,D}\cong \roi_D^\times$, so the map $\cal K_{1,C}\to\cal K_{1,D}$ is surjective; moreover, the sheaves $\cal K_{2,C}$ and $\cal K_{2,D}$ are generated by symbols, and so the map $\cal K_{2,C}\to\cal K_{2,D}$ is also surjective. It follows that $\cal K_{1,(C,D)}\cong\ker(\roi_C^\times\to\roi_D^\times)=:\roi_{(C,D)}^\times$ and that $H^0(C,\cal K_{0,(C,D)})=\ker(H^0(C,\cal K_{0,C})\to H^0(D,\cal K_{0,D}))$. Secondly, $\cal K_{0,C}\cong\bb Z$ via the rank map, and so $H^0(C,\cal K_{0,C})\cong\bb Z$; similarly, $H^0(D,\cal K_{0,D})\cong\bigoplus_{x\in |D|}\bb Z$ via the rank map. If $D\neq 0$, we deduce that the map $H^0(C,\cal K_{0,C})\to H^0(D,\cal K_{0,D})$ is injective and so $H^0(C,\cal K_{0,(C,D)})=0$; while if $D=0$ then evidently $H^0(X,\cal K_{0,(C,D)})=H^0(X,\cal K_{0,C})\cong\bb Z$.

In conclusion, it remains only to construct the cycle class isomorphism \[\CH_0(C;D)\Isoto H^1(C,\roi_{(C,D)}^\times).\] We will do this via a standard Gersten resolution.

Given an open subscheme $U\subseteq C$ containing $|D|$, let $j_U:U\to C$ denote the open inclusion. Then the canonical map $\roi_{(C,D)}^\times\to j_{U*}j_U^*\roi_{C,D}^\times$ fits into an exact sequence of sheaves \[0\To\roi_{(C,D)}^\times\To j_{U*}j_U^*\roi_{(C,D)}^\times\xto{(\op{ord}_x)_x}\bigoplus_{x\in C\setminus U}i_{x*}\bb Z\To0,\] where $i_{x*}\bb Z$ is a skyscraper sheaf at the closed point $x$. This remains exact after taking the filtered colimit over all open $U$ containing $|D|$, yielding \[0\To\roi_{(C,D)}^\times\To k(C)_D^\times \xto{(\op{ord}_x)_x}\bigoplus_{x\in C_0\setminus D}i_{x*}\bb Z\To0,\] where $k(C)_D^\times$ denotes a constant sheaf by abuse of notation. This latter sequence is a flasque resolution of $\roi_{(C,D)}^\times$, and using it to compute cohomology yields a natural isomorphism \[\op{coker}\big(k(C)_D^\times\xto{(\op{ord}_x)_x}\bigoplus_{x\in C_0\setminus D} \bb Z\big)\Isoto H^1(C,\roi_{(C,D)}^\times).\] But the left side of this isomorphism is precisely $\CH_0(C;D)$, thereby completing the proof. 
\end{proof}

\begin{proposition}\label{proposition_cycle_class_for_modulus}
Let $X$ be a smooth variety over a field $k$, and $D$ an effective divisor on $X$. Then the canonical map \[\mbox{\em free abelian group on closed points of }X\setminus D\;\To\; K_0(X,D),\quad x\Mapsto[x]\] descends to a cycle class homomorphism \[\CH_0(X;D)\To K_0(X,D).\]
\end{proposition}
\begin{proof}
We must show that if $C\into X$ is a curve not contained in $|D|$ and $f\in k(C)_D^\times$, then $\sum_{x\in C_0}\op{ord}_x(f)[x]=0$ in $K_0(X,D)$. We will deduce this from Lemma \ref{lemma_class_group_of_curve} once we have  verified a suitable pushforward formalism.

Let $\phi:\tilde C\to C\into X$ be the resulting map from the normalisation $\tilde C$ to $X$, and consider the following pullback square:
\[\xymatrix{
\phi^*D \ar[d]_{\phi'}\ar[r]^{j'} & \tilde C \ar[d]^\phi\\
D \ar[r]_j & X
}\]
We claim that $\phi$ and $j$ are Tor-independent; that is, if $y$ is a closed point of $\tilde C$ such that $x:=\phi(y)$ lies in $|D|$, we must show that $\Tor_{\roi_{X,x}}^i(\roi_{D,x},\roi_{\tilde C,y})=0$ for all $i>0$. But since $D$ is an effective Cartier divisor, there exists a non-zero-divisor $t\in \roi_{X,x}$ such that $\roi_{D,x}=\roi_{X,x}/t\roi_{X,x}$; thus the only possible non-zero higher Tor is $\Tor^1$, which equals the $\phi^*(t)$-torsion of $\roi_{\tilde C,y}$; this could only be non-zero if $\phi^*(t)=0$ in $\roi_{\tilde C,y}$, but this would contradict the condition that $C$ does not lie in $|D|$. This proves the desired Tor-independence.

Moreover, $\phi$ is a finite morphism and $X$ is assumed to be smooth, whence $\phi$ is proper and of finite Tor-dimension. Therefore the projection formula \cite[Prop.~3.18]{Thomason1990} (or \cite[Thm.~4.4]{Coombes1982}) states that the diagram
\[\xymatrix{
K(\tilde C) \ar[d]_{\phi_*}\ar[r]^{j'^*} & K(\phi^*D) \ar[d]^{\phi'_*}\\
K(X) \ar[r]_{j^*} & K(D)
}\]
is well-defined and commutes up to homotopy; so there is an induced pushforward map \[\phi_*:K(\tilde C,\phi^*D)\To K(X,D),\] which by functoriality of pushforwards (as in Section \ref{subsection_LW} we must appeal to \cite[\S4--5]{Coombes1982} to know that the obvious choices of homotopies yield a functorial construction) satisfies $\phi_*[x]=[\phi(x)]$ for any $x\in \tilde C_0$. Therefore
\begin{align*}
\sum_{x\in C_0}\op{ord}_x(f)[x]
	&=\sum_{x\in\tilde C_0}\op{ord}_{\phi(x)}(f)[\phi(x)]\\
	&=\phi_*\big(\sum_{x\in\tilde C_0}\op{ord}_x(f)[x]\big)\\
	&=\phi_*(0)\\
	&=0,
\end{align*}
where $\sum_{x\in\tilde C_0}\op{ord}_x(f)[x]\in K_0(\tilde C,\phi^*D)$ vanishes by Lemma \ref{lemma_class_group_of_curve}.
\end{proof}

\begin{remark}
F.~Binda  \cite{Binda2014} has independently proved Proposition \ref{proposition_cycle_class_for_modulus}, as well as constructing cycle class homomorphisms $\CH_0(X;D;n)\to K_n(X,D)$ for the higher Chow groups with modulus.
\end{remark}

Let $X$ be a $d$-dimensional, smooth variety over $k$. Given effective divisors $D'\ge D$ with the same support, the inclusions $k(C)_{D'}^\times\subseteq k(C)_D^\times$ induce a canonical surjection $\CH_0(X;D')\onto\CH_0(X;D)$. This applies in particular when $D'=rD$ is a thickening of $D$. Combining this observation with Proposition \ref{proposition_cycle_class_for_modulus} we obtain a commutative diagram of inverse systems of Chow groups and relative $K$-groups (recall the definition of $F^dK_0$ from Section \ref{subsection_LW}) in which all maps are surjective (since every group is generated by the closed points of $X\setminus D$):
\[\xymatrix{
F^dK_0(X,D) & F^dK_0(X,2D)\ar@{->>}[l] & F^dK_0(X,3D)\ar@{->>}[l] & F^dK_0(X,4D)\ar@{->>}[l] & \cdots\ar@{->>}[l]\\
\CH_0(X;D)\ar@{->>}[u] & \CH_0(X;2D)\ar@{->>}[l]\ar@{->>}[u] & \CH_0(X;3D)\ar@{->>}[l]\ar@{->>}[u] & \CH_0(X;4D)\ar@{->>}[l]\ar@{->>}[u] & \cdots\ar@{->>}[l]\\
&&\CH_0(X;|D|)\ar@{->>}[ull]\ar@{->>}[ul]\ar@{->>}[u]\ar@{->>}[ur]\ar@{-->>}[urr]&
}\]

There are two natural questions to consider concerning this diagram. Firstly, a question seemingly related to a conjecture of Kerz and Saito \cite[Qu.~V]{KerzSaito2013} is whether the cycle class homomorphism \[\{\CH_0(X;rD)\}_r\To\{F^dK_0(X;rD)\}_r\] is an isomorphism of pro abelian groups, perhaps at least ignoring $(d-1)!$-torsion.

Secondly, changing notation, now suppose that $X'\to X$ is a desingularisation of an integral variety $X$, whose exceptional fibre is an effective Cartier divisor $D$. Then, as a Chow-theoretic analogue of the Bloch--Srinivas conjecture, we ask whether the inverse system \[\CH_0(X';D)\longleftarrow \CH_0(X';2D)\longleftarrow \CH_0(X';3D)\longleftarrow\cdots\] eventually stabilises, with stable value most likely equal to the Levine--Weibel Chow group $\CH_0(X)$ of $X$.

The following theorem simultaneously answers cases of these two questions, working under almost identical hypotheses to Corollary \ref{corollary_Chow_isom}:

\begin{theorem}\label{theorem_modulus}
Let $X$ be a $d$-dimensional, integral, quasi-projective variety over a good field $k$; let $\pi:X'\to X$ be a desingularisation, and $D$ any effective Cartier divisor on $X$ whose support contains the exceptional set. Assume $\op{codim}(X,\pi(|D|))\ge2$ and that the cycle class map $\CH_0(X)\to F^dK_0(X)$ is an isomorphism.

Then $\CH_0(X)\cong\CH_0(X';|D|)$, and the canonical maps \[\CH_0(X';|D|)\To\CH_0(X';rD)\To F^dK_0(X';rD)\] are isomorphisms for $r\gg1$.
\end{theorem}
\begin{proof}
Let $Y\into X$ be the reduced closed subscheme with support $\pi(|D|)$; this has codimension $\ge 2$ and covers $X_\sub{sing}$. Consider the following commutative diagram, which exists for any $r\gg1$:
\[\xymatrix{
\CH_0(X';|D|) \ar[r] & \CH_0(X';rD) \ar[r] & F^dK_0(X';rD) \ar[d]^{\BS_r} \\
\CH_0(X;|Y|) \ar[u] \ar[r] & \CH_0(X) \ar[r] & F^dK_0(X)
}\]
The bottom right horizontal arrow is an isomorphism by assumption; the bottom left horizontal arrow is an isomorphism by Remark \ref{remark_Levine_Weibel}(ii); the left vertical arrow is an isomorphism by Remark \ref{remark_Levine_Weibel}(iv); the right vertical arrow is an isomorphism by Corollary \ref{corollary_Chow_isom}. Since the two top horizontal arrows are surjective, it follows that they are isomorphisms.
\end{proof}

\begin{corollary}\label{corollary_modulus}
Let $X$ be a $d$-dimensional, integral variety over an algebraically closed field $k$ which has strong resolution of singularities; let $\pi:X'\to X$ be a desingularisation, and $D$ any effective Cartier divisor on $X$ whose support contains the exceptional set. Assume $\op{codim}(X,\pi(|D|))\ge2$ and that one of the following is true:
\begin{enumerate}\itemsep0pt
\item $X$ is projective and $\op{char}k=0$; or
\item $X$ is projective and $d\le\op{char}k\neq0$; or
\item $X$ is affine.
\end{enumerate}
Then $\CH_0(X)\cong\CH_0(X';|D|)$, and the canonical maps \[\CH_0(X';|D|)\To\CH_0(X';rD)\To F^dK_0(X';rD)\] are isomorphisms for $r\gg1$.
\end{corollary}
\begin{proof}
This follows from Theorem \ref{theorem_modulus} and the results of Levine and Srinivas recalled in Theorem \ref{theorem_Levine}.
\end{proof}

\begin{remark}[Class field theory of singular varieties]\label{remark_normal_CFT}
In this remark we explain how the $\CH_0$ isomorphism of Theorem \ref{theorem_modulus} over a finite field $\bb F_q$ can be interpreted as part of an unramified class field theory for singular, projective varieties.

Let $X$ be a projective variety over $\bb F_q$ which is regular in codimension one; suppose that a desingularisation $\pi:X'\to X$ exists, that $D$ is an effective Cartier divisor on $X$ whose support contains the exceptional set, and that $\op{codim}(X,\pi(|D|))\ge2$. Write $U=X'\setminus D=X\setminus \pi(|D|)$.

The Kerz--Saito class group \cite{KerzSaito2013} of $U$ is $C(U):=\projlim_r\CH_0(X';rD)$, and their class field theory provides a reciprocity isomorphism $\smash{C(U)^0\isoto\pi_1^\sub{ab}(U)^0}$, where the superscripts $0$ denote degree-$0$ subgroups. Assuming that the conclusions of Theorem \ref{theorem_modulus} are true in this setting, we deduce that $C(U)=\CH_0(X';rD)\cong\CH_0(X)$ for $r\gg1$. They prove moreover that each group $\CH_0(X';rD)^0$ is finite.

In particular, this would prove finiteness of $\CH_0(X)^0$, which is known in the smooth case thanks to the unramified class field theory of S.~Bloch, K.~Kato and Saito, et al. It would also yield a reciprocity isomorphism \[\CH_0(X)^0\Isoto \pi_1^\sub{ab}(U)^0,\quad[x]\mapsto\op{Frob}_x\] However, since the canonical map $\pi_1^\sub{ab}(U)\to \pi_1^\sub{ab}(X)$ is surjective but generally not an isomorphism, we would obtain in general only a surjective reciprocity map \[\CH_0(X)^0\To \pi_1^\sub{ab}(X)^0,\] indicating that the Levine--Weibel Chow group $\CH_0(X)$ is not the correct class group for unramified class field theory of a singular variety.
\end{remark}

\begin{remark}[The case of surfaces]\label{remark_surfaces}
If $X$ is an integral, projective surface over $\bb F_q$ which is regular in codimension one, then we have actually proved the observations of Remark  \ref{remark_normal_CFT} unconditionally: $\CH_0(X)$ is isomorphic to the Kerz--Saito class group $C(X_\sub{reg})$, its degree-$0$ subgroup is finite, and there is a reciprocity isomorphism \[\CH_0(X)^0\Isoto\pi_1^\sub{ab}(X_\sub{reg})^0\] of finite groups. This was brought to the author's attention by \cite{Krishna2015}, in which Krisha reproduced the argument while being unaware of the present paper.

To prove this we must only check that Theorem \ref{theorem_modulus} is true for surfaces over finite fields. In fact, we will let $X$ be a $2$-dimensional, integral, quasi-projective variety over an {\em arbitrary} field $k$ which is regular in codimension one. Then $X$ admits a resolution of singularities $\pi:X'\to X$ with exceptional set equal to exactly $\pi^{-1}(X_\sub{sing})$; let $E:=\pi^{-1}(X_\sub{sing})_\sub{red}$ and $Y:=(X_\sub{sing})_\sub{red}$.

Then Theorem \ref{theorem_BS} is true for the data $X'\to X$, $Y$, $E$. Indeed, it is only necessary to establish the isomorphism (\dag) occurring in the proof, which may be broken into the two isomorphisms
\[\{F^dK_0(X,rY)\}_r\isoto \{F^dK_0(\tilde X,\tilde X\times_XrY)\}_r\isoto\{F^dK_0(X',rE)\}_r,\]
where $\tilde X\to X$ denotes the normalisation of $X$. The second of these isomorphisms is due to Krishna and Srinivas \cite[Thm.~1.1]{Krishna2002}; the first isomorphism follows from the isomorphism $\{K_0(X,rY)\}_r\isoto \{K_0(\tilde X,\tilde X\times_XrY)\}_r$, which is a case of the author's pro-excision theorem \cite[Corol.~0.4 \& E.g.~2.5]{Morrow_pro_H_unitality}, and the obvious surjectivity just as in the proof of Theorem \ref{theorem_BS}.

Now assume further (perhaps after blowing-up $X'$ at finitely many points) that there is an effective divisor $D$ on $X'$ with support $\pi^{-1}(X_\sub{sing})$. Since the cycle class map $\CH_0(X)\to F^dK_0(X)$ is automatically an isomorphism (as we remarked immediately before Theorem \ref{theorem_Levine}), it follows that the assertions of Theorem \ref{theorem_modulus} are also true, as required: $\CH_0(X)\cong\CH_0(X';|D|)$, and the canonical maps $\CH_0(X';|D|)\to\CH_0(X';rD)\to F^dK_0(X';rD)$ are isomorphisms for $r\gg1$.
\end{remark}

%\bibliographystyle{acm}
%\bibliography{../Bibliography}

\def\cprime{$'$}

\noindent Matthew Morrow\hfill {\tt morrow@math.uni-bonn.de}\\
Mathematisches Institut\hfill {\tt http://www.math.uni-bonn.de/people/morrow/}\\\
Universit\"at Bonn\\
Endenicher Allee 60\\
53115 Bonn, Germany
\end{document}